\documentclass[10pt]{amsart}
\usepackage{latexsym, amsmath, amsfonts, amssymb, mathrsfs, fancyhdr, tikz, tikz-cd}
\usepackage{verbatim}
\usepackage{multicol}
\usetikzlibrary{matrix,arrows,decorations.pathmorphing,calc}

\makeatletter
\@namedef{subjclassname@2020}{%
  \textup{2020} Mathematics Subject Classification}
\makeatother

\def\ds{\displaystyle}

\def\R{\mathbb{R}}

\def\e{\varepsilon}
\def\s{\sigma}

\def\d{\delta}

\def\l{\lambda}

\def\S{\mathcal{S}}

\def\la{\langle}
\def\ra{\rangle}

\def\c{\mathfrak{c}}
\def\cR{\mathcal{R}}

\def\ch2{\mathbb{C} \mathbb{H}^2}
\def\h2{\mathbb{H}^2}

\def\ddt{\frac{\partial}{\partial \theta}}
\def\ddr{\frac{\partial}{\partial r}}
\def\th{\theta}

\def\oshr/2{\cosh \left( \frac{r}{2} \right)}
\def\inhr/2{\sinh \left( \frac{r}{2} \right)}
\def\osh2r/2{\cosh^2 \left( \frac{r}{2} \right)}
\def\inh2r/2{\sinh^2 \left( \frac{r}{2} \right)}
\def\-1/4{- \frac{1}{4}}
\def\H{\mathbb{H}}
\def\C{\mathbb{C}}
\def\S{\mathbb{S}}
\def\P{\mathbb{P}}

\def\L{\Lambda}

\newtheorem{theorem}{Theorem}[section]
\newtheorem{lemma}[theorem]{Lemma}

\theoremstyle{definition}

\theoremstyle{remark}
\newtheorem{remark}[theorem]{Remark}

\numberwithin{equation}{section}

\begin{document}

\title{Curvature Operators on K\"{a}hler Manifolds}

\author{Barry Minemyer}
\address{Department of Mathematics, Computer Science, and Digital Forensics, Commonwealth University - Bloomsburg, Bloomsburg, Pennsylvania 17815}
\email{bminemyer@commonwealthu.edu}


\subjclass[2020]{Primary 53C24, 53C35, 53C55; Secondary 53C15, 53C20, 53C21}

\date{\today.}



\begin{abstract}
We prove that there exist K\"{a}hler manifolds that are not homotopy equivalent to a quotient of complex hyperbolic space but which admit a Riemannian metric with nonpositive curvature operator.  
This shows that K\"{a}hler manifolds do not satisfy the same type of rigidity with respect to the curvature operator as quaternionic hyperbolic and Cayley hyperbolic manifolds and are thus more similar to real hyperbolic manifolds in this setting.  
Along the way we also calculate explicit values for the eigenvalues of the curvature operator with respect to the standard complex hyperbolic metric.  
\end{abstract}

\maketitle



\section{Introduction}\label{Section:Introduction}

Let $(M, g)$ be a Riemannian manifold, $p \in M$, and let $\L^2(T_p M)$ denote the space of alternating 2-forms on $T_p M$, the tangent space to $M$ at $p$.  
The Riemannian metric $g$ gives rise an inner product $\left\la \, , \right\ra$ on $\L^2(T_p M)$ as follows.  
Given $U, V, W, X \in T_p M$ define 
	\begin{equation}\label{eqn:curv op inner product}
	\left\la U \wedge V, W \wedge X \right\ra = g(U,W) g(V,X) - g(U,X) g(V,W) 
	\end{equation}
and extend linearly to all of $\L^2(T_p M)$.  
The curvature operator $\cR$ on $M$ with respect to $g$ is the unique endomorphism on $\L^2 (T_p M)$ defined by 
	\begin{equation*}
	\left\la \cR(U \wedge V), W \wedge X \right\ra = g (R(U,V)W, X) 
	\end{equation*}
where $R$ is the sectional curvature tensor corresponding to the metric $g$.  

Being symmetric, the curvature operator $\cR$ has all real eigenvalues.  
The sign of $\cR$ is then defined to be the sign of its eigenvalues.  
In this paper we will be concerned with nonpositive curvature operators.  
But there is also a large amount of literature concerning negative, positive, and nonnegative curvature operators.  
For more information one should consult the excellent survey article by C.S. Aravinda \cite{Aravinda}, where much of the content of this Introduction was obtained as well.  

One easily checks that the curvature operator is stronger than that of sectional curvature.  
Indeed, if a Riemannian metric has nonpositive (respectively negative, positive, nonnegative) curvature operator at every point in the manifold, then it is an easy exercise to see that all sectional curvatures are nonpositive (respectively negative, positive, nonnegative).  
The reverse implication is generally false, as we will see below.  

In the quaternionic hyperbolic and Cayley hyperbolic setting, manifolds which admit metrics with nonpositive curvature operator are extremely rigid.  
A result of Corlette \cite{Corlette} says the following.  
Suppose $M$ is a compact quaternionic hyperbolic (with quaternionic dimension at least 2) or Cayley hyperbolic manifold, and supposed that $N$ is a closed Riemannian manifold with nonpositive curvature operator whose fundamental group is isomorphic to the fundamental group of $M$.  
Then $M$ and $N$ are isometric (up to scaling the metric).  
From this result one sees that the negatively curved exotic manifolds constructed by Aravinda and Farrell in \cite{AF1} and \cite{AF2} do not admit Riemannian metrics with nonpositive curvature operator (see also \cite{AF}).  

In contrast, there is a large amount of flexibility in the real hyperbolic setting.  
If all sectional curvatures are pinched arbitrarily close to a given nonzero number, then the curvature operator must have the same sign as this number.  
So, in particular, the pinched metrics constructed by Gromov and Thurston in \cite{GT} and, more generally, the metrics produced by Ontaneda's smooth hyperbolization \cite{Ontaneda} all have negative curvature operator but are not diffeomorphic to a hyperbolic manifold.  
Additionally, an inspection of the metric constructed in \cite{GT} shows that all eigenvalues of the curvature operator are also pinched near $-1$.  

In the present paper we address the complex hyperbolic setting.  
No rigidity results mirroring Corlette's result are known, but there are also no examples of the ``flexibility" demonstrated by the Gromov and Thurston manifolds.  
This is likely due to the scarcity of examples of negatively curved K\"{a}hler manifolds that are not diffeomorphic to quotients of complex hyperbolic space $\C \H^n$ and, within these few examples, the difficulty in computing the eigenvalues of their curvature operator.  
Prior to a recent breakthrough of Stover and Toledo \cite{ST}, the only examples of such manifolds that the author is aware of which admit a nonpositively curved K\"{a}hler metric are due to Mostow and Siu \cite{MS}, Deraux \cite{Deraux}, Hirzebruch \cite{Hirzebruch} (whose K\"{a}hler metric was constructed by Zheng \cite{Zheng}), Zheng \cite{Zheng2}, and the exotic examples of Farrell and Jones \cite{FJ}.  
Our main result states that the manifolds constructed in \cite{ST} admit Riemannian metrics with nonpositive curvature operator.


\begin{theorem}\label{thm:main theorem}
For each complex dimension $n$ there exist K\"{a}hler manifolds $M$ of dimension $n$ which admit a Riemannian metric with nonpositive curvature operator and are not homotopy equivalent to a quotient of $\C \H^n$.  
\end{theorem}

The Riemannian metric from Theorem \ref{thm:main theorem} is constructed in \cite{Min CH}, where it is also shown that the metric can be constructed so that all sectional curvatures lie in $[-4-\e, -1+\e]$ for any prescribed $\e > 0$.  
In the present article we prove that the curvature operator of this metric is nonpositive.  
Let us quickly note that this Riemannian metric is not K\"{a}hler.  The fact that these manifolds admit a (negatively curved) K\"{a}hler metric is due to Zheng \cite{Zheng}, and Stover and Toledo \cite{ST} proved that these manifolds are not homotopy equivalent to a quotient of $\C \H^n$.  

Theorem \ref{thm:main theorem} shows that, with respect to the curvature operator, the K\"{a}hler setting is more like the real hyperbolic setting than the quaternionic or Cayley hyperbolic situations.
It should be noted that this is at least a little surprising.  
Mostow rigidity \cite{Mostow} applies to all of these situations of course (for appropriate dimensions).  
But results of Hernandez \cite{Hernandez}, Yau and Zheng \cite{YZ}, and Deraux and Seshadri \cite{DS} show that, with respect to sectional curvature, negatively curved K\"{a}hler manifolds are overall much more rigid than general negatively curved Riemannian manifolds.  
But, unlike the quaternionic and Cayley hyperbolic situation, there is at least some flexibility with the curvature operator in the K\"{a}hler setting.  

Our work in this paper also gives an easy proof of the following.

\begin{theorem}\label{thm:main cor}
The manifolds from Theorem \ref{thm:main theorem} also admit a Riemannian metric with negative sectional curvature, but whose curvature operator is not nonpositive.  
\end{theorem}

There is nothing surprising about Theorem \ref{thm:main cor}.  
The Riemannian metric constructed in Theorem \ref{thm:main theorem} is equal to the standard complex hyperbolic metric $\c_n$ on parts of the manifold.  
The metric $\c_n$ has (many) eigenvalues of 0 (see below), and so intuitively one should be able to slightly perturb the metric in the direction of a 0-eigenspace in order to create a positive eigenvalue for the curvature operator.  
Since all sectional curvatures of $\C \H^n$ lie in the interval $[-4,-1]$, a sufficiently small perturbation should maintain negative sectional curvature.  
We quickly give a precise description of how this can be done in Section 4.  

Lastly, the calculations in Section 3 give an easy description of the eigenvalues of the curvature operator for the standard complex hyperbolic metric.  
One would expect that these are already known and, in any case, can be calculated using the equation directly preceeding Proposition IX.7.2 of \cite{KN} and the curvature formulas in \cite{Min warped}.  
But the author is unaware of anywhere in the literature where these are explicitly written down.  
So we record these eigenvalues here.

\begin{theorem}\label{thm:ch eigenvalues}
The curvature operator $\cR$ corresponding to the standard complex hyperbolic metric $\c_n$ has eigenvalues
	\begin{itemize}
	\item  $0$ with multiplicity $n^2 - n$.  
	\item  $-2$ with multiplicity $n^2 - 1$.  
	\item  $-(2n+2)$ with multiplicity 1.  
	\end{itemize}
Moreover, the eigenspace corresponding to the last eigenvalue of $-(2n+2)$ is equal to the span of the vector
	$$ \sum Y_i \wedge Y_{i'} $$
where $(i, i')$ form a holomorphic pair, and this sum ranges over all holomorphic pairs in an appropriately chosen basis (equation \eqref{eqn:ON basis} below).  
\end{theorem}

\begin{remark}
Note that, if $M$ has complex dimension $n$, then the real dimension of $\L^2(T_p M)$ is
	\begin{equation*}
	{2n \choose 2} = \frac{2n(2n-1)}{2} = 2n^2 - n.
	\end{equation*}
Therefore, the multiplicities stated in Theorem \ref{thm:ch eigenvalues} add up correctly.  
\end{remark}

The eigenvalues for the curvature operator of the complex projective space $\C \P^n$ can be found in \cite{BK} (Section 5.2).  
Note that the eigenvalues of the complex hyperbolic and complex projective metrics are not exactly negatives of each other, but they are very similar.  
The author was surprised to see that one of the eigenvalues of $\cR$ with respect to $\c_n$ depends on $n$.  
But this makes at least some sense intuitively.  
As $n$ increases, $\cR$ gains more eigenvalues of 0.  
But all sectional curvatures of $\c_n$ remain in $[-4,-1]$.  
Thus, there must be either a deccreasing eigenvalue or an increasing multiplicity (or both) to offset the extra nonnegative eigenvalues.  

This paper is laid out as follows.  
In Section \ref{section2: polar coordinates} we review the metric $\c_n$ written in polar coordinates about a complex codimension 1 submanifold.  
The results from this Section can almost all be found in \cite{Bel complex}, \cite{Min warped}, and \cite{Min CH}.  
Section \ref{section3: eigenvalues} is the meat of the paper.  
Here we directly calculate or approximate the eigenvalues of $\c_n$ but as functions of the structure constants $c_i$ defined via equation \eqref{eqn:structure constants 2} below.  
When $c_i = 2$ for all $i$ we recover $\c_n$, and so the calculations in Section \ref{section3: eigenvalues} prove Theorem \ref{thm:ch eigenvalues}.  
In Section \ref{section4: proof of main theorem} we review the construction of the metric from \cite{Min CH} and prove that it has nonpositive curvature operator.  
We also indicate how to vary the metric to satisfy Theorem \ref{thm:main cor}.

\subsection*{Acknowledgments}
The author is grateful to C.S. Aravinda for mentioning this problem via email the day after \cite{Min CH} was posted to arXiv, for forwarding reference \cite{Aravinda}, and for helpful comments on the first draft of this paper.

\vskip 20pt

\section{Curvature formulas for the complex hyperbolic metric written in polar coordinates about a complex hyperplane}\label{section2: polar coordinates}
In this Section we review curvature formulas for $\C \H^n$ written in polar coordinates about a copy of $\C \H^{n-1}$, and we give a slightly new version of the curvature formulas for the corresponding warped product metric.  
Curvature formulas for the pair $(\C \H^n, \C \H^{n-1})$ were derived in \cite{Bel complex} for curvatures in $[ -1, - 1/4]$, and converted to curvatures in $[-4,-1]$ in \cite{Min warped} and \cite{Min CH} in slightly different settings.  
All besides one of the curvature formulas in Theorem \ref{thm:new curvature equations} below can be found in the previous references and, in particular, this Section is very similar to Section 2 of \cite{Min CH}.

\subsection{The metric on $\C \H^n$ in polar coordinates about $\C \H^{n-1}$}\label{subsection:chn/chn-1}
Let $\C \H^{n-1}$ denote a complex codimension one totally geodesic complex submanifold in $\C \H^n$, let $r$ denote the distance to $\C \H^{n-1}$ within $\C \H^n$, and let $\c_n$ and $\c_{n-1}$ denote the metrics on $\C \H^n$ and $\C \H^{n-1}$ normalized to have constant holomorphic curvature $-4$.  
Let $\phi : \C \H^n \to \C \H^{n-1}$ denote the orthogonal projection onto $\C \H^{n-1}$, and let $E(r)$ denote the $r$-tube about $\C \H^{n-1}$.   
Since $\C \H^{n-1}$ is contractible, topologically one has that  $E(r) \cong \R^{2n-2} \times \S^1$ and $\C \H^n \setminus \C \H^{n-1} \cong \R^{2n-2} \times \S^1 \times (0, \infty)$.  
The metric in $\C \H^n$ in polar coordinates about $\C \H^{n-1}$ is given by
	\begin{equation}\label{eqn:c metric}
	\c_n = \cosh^2(r) \c_{n-1} + \frac{1}{4} \sinh^2(2r) d \th^2 + dr^2
	\end{equation}
where $d \th$ denotes the standard metric on $\S^1$ and $\c_{n-1}$ is the complex hyperbolic metric on the orthogonal complement to $\text{span}(d/d \th , d/dr)$.  
As described below, this subspace will not be tangent to $E(r)$.
In this paper we will frequently refer to this subspace as the {\it horizontal distribution}, while we call $\text{span}(d/d \th , d/dr)$ the {\it vertical distribution}.  

Let $p \in \C \H^{n-1}$.  
We define a special basis $(\check{X}_1, \check{X}_2, \hdots, \check{X}_{2n-2})$ of $T_p \C \H^{n-1}$, which we call a {\it holomorphic basis} near $p$, as follows.  
We first define $\check{X}_1$ to be any unit vector in $T_p \C \H^{n-1}$.  
We then define $\check{X}_2 = J \check{X}_1$, where $J$ denotes the complex structure on $\C \H^n$.  
Since $\C \H^{n-1}$ is a complex submanifold, $J \check{X}_1 \in T_p \C \H^{n-1}$.  
We call such a pair $\{ \check{X}, J \check{X} \}$ a {\it holomorphic pair}.
Let $\check{X}_3$ be any unit vector in $T_p \C \H^{n-1}$ which is orthogonal to $\text{span}(\check{X}_1, \check{X}_2)$, and let $\check{X}_4 = J \check{X}_3$.  
It is easily seen that $J \check{X}_3$ is orthogonal to the span of $\{ \check{X_1}, \check{X}_2, \check{X}_3 \}$.  
We continue in this way to construct an orthonormal basis $(\check{X}_1, \check{X}_2, \hdots, \check{X}_{2n-2})$ of $T_p \C \H^{n-1}$ which satisfies that, for $i$ odd, the pair $(\check{X}_i, \check{X}_{i+1})$ is a holomorphic pair.
Via a standard construction, we can extend this basis to a neighborhood of $p$ in $\C \H^{n-1}$ in such a way that $[ \check{X}_i, \check{X}_j ]_p = 0$ for all $i$ and $j$.  
Let us note that if $i$ is an odd integer then $\text{exp}_p(\text{span}(\check{X}_i, \check{X}_{i+1}))$ is a complex line and thus the sectional curvature with respect to $\c_n$ of this 2-plane is $-4$, whereas if $\{ \check{X}_i, \check{X}_j \}$ do not form a holomorphic pair then $\text{exp}_p(\text{span}(\check{X}_i, \check{X}_j))$ is a totally real totally geodesic subspace of $\C \H^{n-1}$ and thus has curvature $-1$ with respect to $\c_n$.   

Let $q \in \C \H^n$ be such that $\phi(q) = p$.
Extend the collection $( \check{X}_i )_{i=1}^{2n-2}$ to vector fields $X_1, X_2, \hdots, X_{2n-2}$ defined near $q$ in $\C \H^n$ via $d \phi^{-1}$ in such a way that these vector fields are invariant with respect to both $\theta$ and $r$.  
We will call such a frame a {\it holomorphic frame} near $q$.  
We need to understand the Lie brackets of this frame associated to the metric in \eqref{eqn:c metric}. 
It is proved in \cite{Bel complex} that there exist {\it structure constants} $c_{i}$ such that
	\begin{equation}\label{eqn:structure constants 1}
	[X_i,X_{i+1}] = c_i \ddt \qquad \text{for } i \text{ odd} 
	\end{equation}
and that $[X_i,X_j] = 0$ whenever $\{ X_i, X_j \}$ is not a holomorphic pair.  
Moreover, in \cite{Bel complex} and \cite{Min warped} it is actually proved that, with respect to the complex hyperbolic metric $\c_n$, 
	\begin{equation}\label{eqn:structure constants 2}
	c_i = 2 \qquad \forall i.  
	\end{equation}

Let $v(r)$ and $h(r)$ be positive real-valued functions of $r$ and define $E = \R^{2n-2} \times \S^1$.
Define the warped product metric $\mu := \mu_{v,h}$ on $E \times (0,\infty)$ by
	\begin{equation}\label{eqn:mu}
	\mu =  h^2 \c_{n-1} + \frac{1}{4} v^2 d \th^2 + dr^2.
	\end{equation}
Note that when $v = \sinh(2r)$ and $h = \cosh(r)$ we recover the complex hyperbolic metric $\c_n$.
Fix a holomorphic frame $(X_i)_{i=1}^{2n-2}$ as above.  
Let $X_{2n-1} = d/d \th$ and $X_{2n} = d/ dr$. 
Define the following orthonormal frame for $\mu$:
	\begin{equation}
	Y_i = \frac{1}{h} X_i \; \text{ for } 1 \leq i \leq 2n-2 \hskip 30pt  Y_{2n-1} = \frac{1}{\frac{1}{2} v} X_{2n-1}  \hskip 30pt Y_{2n} = X_{2n}.  \label{eqn:ON basis}
	\end{equation}
Formulas for the components of the $(4,0)$ curvature tensor $R^\mu$ of $\mu$ as functions of $h$, $v$, and $c_i$ for $i = 1, 3, \hdots , 2n-3$ are given by the following Theorem.


\begin{theorem}[compare Sections 7 and 8 of \cite{Bel complex}, Theorem 4.3 of \cite{Min warped}, and Theorem 2.2 of \cite{Min CH}]\label{thm:new curvature equations}
Let $R^{\mu}_{i,j,k,l} := \mu( R^\mu (Y_i, Y_j)Y_k, Y_l )$, and let $(X_i)_{i=1}^{2n-2}$ be a holomorphic basis.
Then, up to the symmetries of the curvature tensor, formulas for all nonzero components of the (4,0) curvature tensor $R^\mu$ are:
	\begin{align*}
	&R^{\mu}_{i,j,i,j} = - \frac{1}{h^2} - \left( \frac{h'}{h} \right)^2  \hskip 40pt  R^{\mu}_{i,2n-1,i,2n-1} = - \frac{h' v'}{hv} + \frac{c_i^2 v^2}{16 h^4}  \\	
	&R^{\mu}_{i, i+1, i, i+1} = - \left( \frac{h'}{h} \right)^2 - \frac{4}{h^2} - \frac{3 c_i^2 v^2}{16 h^4} \\
	&R^{\mu}_{i,2n,i,2n} = - \frac{h''}{h}   \hskip 70pt  R^{\mu}_{2n-1,2n,2n-1,2n}= - \frac{v''}{v}  \\
	&R^{\mu}_{i, i+1, 2n-1, 2n} = 2 R^{\mu}_{i, 2n-1, i+1, 2n} = -2R^{\mu}_{i, 2n, i+1, 2n-1} = - c_{i} \frac{v}{2h^2} \left( \ln \frac{v}{h} \right)'  \\
	&R^{\mu}_{i, i+1, k, k+1} = 2R^{\mu}_{i, k, i+1, k+1} = -2R^{\mu}_{i, k+1, i+1, k} = - \frac{2}{h^2} - \frac{c_i c_k v^2}{8h^4}  
	\end{align*}
where $1 \leq i, j, k \leq 2n-2$, $k$ is an odd integer different from $i$, and $j \neq i, i+1$.  
Also, any equations using both $i$ and $i+1$ assumes that $i$ is an odd integer.
\end{theorem}

\vskip 10pt

Let us give a quick remark about the citation here.  
These formulas were originally computed by Belegradek in \cite{Bel complex} but for $c_i = 2$ for all $i$ (and with a slightly different frame, and with curvature scaled to $[-1,-1/4]$).  
The author needed these formulas with variable inputs for $c_i$, and so these were recalculated in \cite{Min CH}.  
Neither of these papers needed the last mixed term at the bottom ($\ds{ R^{\mu}_{i, i+1, k, k+1} }$), but it is needed here in order to compute the eigenvalues of the associated curvature operator.  
This mixed term was computed in \cite{Min warped} but for $c_i = c_k = 2$.  
For brevity we do not prove this one formula here.  
But one can obtain this formula in a straightforward way by combining the methods of \cite{Min warped} and \cite{Min CH}. 

The following Lemma is proved in \cite{Min CH}.  
We list it here in order to reference it in a few places later in the paper.

\begin{lemma}[See Section 2 of \cite{Min CH}]\label{thm:pinched curvature}
Let $g$ denote the metric $\mu$ from equation \eqref{eqn:mu} on $X = \R^{n-2} \times \S^1 \times (0,\infty)$ with $h(r) = \cosh(r)$ and $v(r) = \sinh(2r) = 2 \sinh(r) \cosh(r)$.  
Let $q = (p, \theta, r) \in X$.  
Then for all $\e > 0$, there exists $R > 0$ such that for all 2-planes $\sigma \subseteq T_q (X)$ where $r > R$, we have that $K_g(\s) \in (-4-\e, -1+\e)$ provided that all structure constants $c_i$ from equation \eqref{eqn:structure constants 1} satisfy $c_i \in [-2, 2]$.  
\end{lemma}


The {\it integrable complex hyperbolic metric} $g_{I}$ is defined by setting $h(r) = \cosh(r)$ and $v(r) = \sinh(2r) = 2 \sinh(r) \cosh(r)$ in equation \eqref{eqn:mu}, and setting all structure constants defined in \eqref{eqn:structure constants 1} identically equal to zero.  
This is the metric one would obtain if the complex hyperbolic metric were integrable, that is, if the horizontal fiber $d \phi^{-1} (T_p \C \H^{n-1})$ were always tangent to $E(r)$.  
So, as a metric, we have that
	\begin{equation*}
	g_{I} = \cosh^2(r) \c_{n-1} + \frac{1}{4} \sinh^2(2r) d \th^2 + dr^2
	\end{equation*}
but with different structure constants than the metric $\c_n$.  

By Lemma \ref{thm:pinched curvature} we know that, given $\e > 0$, all sectional curvatures of $g_I$ lie in $(-4-\e, -1+\e)$ for $r$ sufficiently large. 
But more important for our purposes is that, for $r$ sufficiently large, all mixed terms of the curvature tensor are approximately 0 (see the equations in Theorem \ref{thm:new curvature equations}).  
Thus, for $r$ large the basis $(Y_i)$ (approximately) diagonalizes the curvature operator $\cR$ with respect to $g_I$, and the eigenvalues of $\cR$ are (approximately) the sectional curvatures of the coordinate planes.  
We will give explicit values for these eigenvalues in Section \ref{section3: eigenvalues}.


\vskip 20pt

\section{Eigenvalues for the curvature operator associated to $\mu$}\label{section3: eigenvalues}
In this Section we analyze the eigenvalues of the curvature operator $\cR$ with respect to the metric $\mu$ from equation \eqref{eqn:mu} defined on $\R^{2n-2} \times \S^1 \times (0, \infty)$.  
For the remainder of this Section we only consider the warping functions $h(r) = \cosh(r)$ and $v(r) = \sinh(2r)$ for the metric $\mu$.  
Note that when $c_i = 2$ for all $i$ this will yield the eigenvalues for the curvature operator with respect to the complex hyperbolic metric $\c_n$, and when $c_i = 0$ for all $i$ this will approximate the eigenvalues for the curvature operator of the integrable complex hyperbolic metric $g_I$ for $r$ large.  

It is instructive to first consider the cases for complex dimensions $n=2$ and $3$.  
From this work it will be easy to state and prove the results for general $n$.

\subsection{The n=2 case}  The metric $\mu$ is defined on the space $X = \R^2 \times \S^1 \times (0, \infty)$.  
Let $q = (p, \theta, r) \in X$, and let $(\check{X}_1, \check{X}_2)$ be a holomorphic basis near $p$.  
Extend this basis to a holomorphic frame $(X_1, X_2)$ about $q$ as described in Section \ref{section2: polar coordinates}.  
Let $X_3 = d/d \th$ and $X_4 = d/dr$.  
Note that the only nontrivial Lie bracket within this basis is $[X_1, X_2] = c_1 (d/d \th)$.  

Let
	\begin{equation*}
	Y_1 = \frac{1}{\cosh(r)} X_1 \qquad Y_2 = \frac{1}{\cosh(r)} X_2 \qquad Y_3 = \frac{1}{\sinh(r) \cosh(r)} X_3 \qquad Y_4 = X_4.
	\end{equation*}
Then $(Y_1, Y_2, Y_3, Y_4)$ is an orthonormal basis for $\mu$ and, consequently, 
	\begin{equation*}
	(Y_1 \wedge Y_2, Y_1 \wedge Y_3, Y_1 \wedge Y_4, Y_2 \wedge Y_3, Y_2 \wedge Y_4, Y_3 \wedge Y_4)
	\end{equation*}
is an orthonormal basis for $\Lambda^2(T_q M)$ with respect to the inner product described in \eqref{eqn:curv op inner product}.  

By reordering this basis appropriately we can obtain a block-diagonal representation of $\cR$.  
This reordering is as follows.  
We first list the basis vectors that correspond to holomorphic pairs.  
So, in this case, we start by listing $Y_1 \wedge Y_2$ and $Y_3 \wedge Y_4$.  
Then, given a vector of the form $Y_i \wedge Y_j$, we pair it with the vector $Y_{i'} \wedge Y_{j'}$ where $\{ i, i' \}$ and $\{ j, j' \}$ are each holomorphic pairs.  
This gives us the ordered basis
	\begin{equation*}
	\beta = (Y_1 \wedge Y_2, Y_3 \wedge Y_4, Y_1 \wedge Y_3, Y_2 \wedge Y_4, Y_1 \wedge Y_4, Y_2 \wedge Y_3).  
	\end{equation*}
With respect to $\beta$, the matrix representation of $\cR$ can be approximated for $r > > 0$ by
	\begin{equation}
	[\cR]_{\beta} = \begin{bmatrix}
	-1-\frac{3}{4}c_1^2 & -c_1 & 0 & 0 & 0 & 0  \\
	-c_1 & -4 & 0 & 0 & 0 & 0  \\
	0 & 0 & -2 + \frac{1}{4} c_1^2 & - \frac{1}{2} c_1 & 0 & 0  \\
	0 & 0 & - \frac{1}{2} c_1 & -1  & 0 & 0  \\
	0 & 0 & 0 & 0 & -1 & \frac{1}{2} c_1  \\
	0 & 0 & 0 & 0 & \frac{1}{2} c_1 & -2 + \frac{1}{4} c_1^2
	\end{bmatrix}
	\end{equation}
Note that this matrix gives the exact values for the curvature operator when $c_1 = 2$, and is an approximation for $r$ large when $c_1 \in [0,2)$.  

The first block in the top-left corner of the matrix corresponds to the bivectors that form a holomorphic pair.  
We will refer to this block as the {\it holomorphic block} and denote it by $H_2$.  
One sees immediately that this block is negative definite:  the (1,1) entry $-1 - \frac{3}{4} c_1^2$ is negative over $c_1 \in [0,2]$, and the determinant 
	\begin{equation}\label{eqn:detS2}
	\text{det}(H_2) = 4 + 2c_1^2
	\end{equation}
is positive over this same domain.  
Therefore, the eigenvalues of $\cR$ corresponding to the holomorphic block are negative, and this holds for all values of $c_1$ provided $r$ is chosen sufficiently large.  
A direct calculation shows that the eigenvalues for the holomorphic block are -6 and -2 when $c_1 = 2$.  
The eigenspace associated to the eigenvalue $-6$ when $c_1 = 2$ is the span of the vector $\begin{bmatrix} 1 & 1 & 0 & 0 & 0 & 0 \end{bmatrix} = (Y_1 \wedge Y_2) + (Y_3 \wedge Y_4)$.  
When $c_1 = 0$ the eigenvalues of $H_2$ are $-1$ and $-4$.  
So these values approximate two of the eigenvalues of the curvature operator of $g_I$ for $r$ large.  

The remaining two blocks are identical up to reordering the basis vectors and correspond to a pair of bivectors of the form $Y_i \wedge Y_j$ where one vector is horizontal and the other vector is vertical.  
The characteristic polynomial for these blocks factors as 
	\begin{equation}\label{eqn:block 1}
	\left( \l + 2 \right) \left( \l + \left( 1-\frac{1}{4}c_1^2 \right) \right)
	\end{equation}
and so the eigenvalues are -2 and $-1 + \frac{1}{4} c_1^2$.  
For $c_1 = 2$ this gives eigenvalues of -2 and 0 (each with multiplicity 2) for the curvature operator of the complex hyperbolic metric.
When $c_1 = 0$ this gives (approximate) eigenvalues of $-2$ and $-1$ for the curvature operator associated to $g_I$.  

Since this block has an eigenvalue of 0 when $c_1 = 2$, approximating this for $r$ large is not sufficient to guarantee that all eigenvalues are nonpositive for all $c_1 \in [0,2]$.  
If one plugs in the actual values from Theorem \ref{thm:new curvature equations} into this block, they obtain
	\begin{equation*}
	\begin{bmatrix}
	-1 - \frac{(4-c_1^2) \sinh^2(r)}{4 \cosh^2(r)} & - \frac{1}{2} c_1 \\
	- \frac{1}{2} c_1 & -1
	\end{bmatrix}
	\end{equation*}
The characteristic polynomial for this matrix is
	$$ \l^2 + \left( 2 + \frac{(4-c_1^2) \sinh^2(r)}{4 \cosh^2(r)} \right) \l + \left( 1 + \frac{(4-c_1^2) \sinh^2(r)}{4 \cosh^2(r)} - \frac{1}{4} c_1^2 \right) $$
and the zeros of this polynomial are
	\begin{equation*}
	\l = \frac{1}{2} \left( -2 - \frac{4-c_1^2}{4} \tanh^2(r) \pm \sqrt{\frac{(4-c_1^2)^2}{16} \tanh^4(r) + c_1^2} \right)
	\end{equation*}
One recoups the eigenvalues of -2 and 0 when $c_1 = 2$, and for $c_1 \in [0,2)$ the triangle inequality shows that both zeros are negative.

\subsection{The n=3 case}  The metric $\mu$ is defined on the space $X = \R^4 \times \S^1 \times (0, \infty)$.  
In the exact same way as above, let $q = (p, \theta, r) \in X$, and let $(\check{X}_1, \check{X}_2, \check{X}_3, \check{X}_4)$ be a holomorphic basis near $p$.  
Extend this basis to a holomorphic frame $(X_1, X_2, X_3, X_4)$ about $q$.  
Note that the holomorphic pairs are $(X_1, X_2)$ and $(X_3, X_4)$.  
Let $X_5 = d/d \th$ and $X_6 = d/dr$.  
Note that the only nontrivial Lie brackets within this basis are $[X_1, X_2] = c_1 X_5$ and $[X_3, X_4] = c_3 X_5$.  

Let
	\begin{equation*}
	Y_i = \frac{1}{\cosh(r)} X_i \text{ for i=1, 2, 3, 4} \qquad Y_5 = \frac{1}{\sinh(r) \cosh(r)} X_5 \qquad Y_6 = X_6.
	\end{equation*}
Then $(Y_i)_{i=1}^6$ is an orthonormal basis for $\mu$.  
Using the same ordering as the $n=2$ case, we obtain the following  basis
	\begin{align*}
	\beta = (&\underline{Y_1 \wedge Y_2, Y_3 \wedge Y_4, Y_5 \wedge Y_6}, \overline{Y_1 \wedge Y_3, Y_2 \wedge Y_4}, \underline{Y_1 \wedge Y_4, Y_2 \wedge Y_3}, \overline{Y_1 \wedge Y_5, Y_2 \wedge Y_6},  \\
	&\underline{Y_1 \wedge Y_6, Y_2 \wedge Y_5}, \overline{Y_3 \wedge Y_5, Y_4 \wedge Y_6}, \underline{Y_3 \wedge Y_6, Y_4 \wedge Y_5})
	\end{align*}
which is an orthonormal basis for $\Lambda^2(T_q M)$ with respect to the metric \eqref{eqn:curv op inner product}.  
The under and over lines in the basis are just to indicate which collections of bivectors will correspond to blocks along the diagonal of the matrix representation for $\cR$.  
Note that the first block is the collection of holomorphic pairs, and the remaining blocks consist of two bivectors where an index from each bivector forms a holomorphic pair.  

For $r$ large, the holomorphic block can be approximated by
	\begin{equation*}
	H_3 = \begin{bmatrix}
	-1 - \frac{3}{4}c_1^2 & -\frac{1}{2} c_1 c_3 & -c_1  \\
	- \frac{1}{2} c_1 c_3 & -1 - \frac{3}{4} c_3^2 & -c_3  \\
	-c_1 & -c_3 & -4
	\end{bmatrix}
	\end{equation*}
One can see that this matrix is negative definite by induction. 
The lower-right $2 \times 2$ matrix is the holomorphic block in the $n=2$ case (with $c_3$ in place of $c_1$) and therefore is negative definite.  
The determinant of $H_3$ is
	\begin{equation}\label{eqn:detS3}
	\text{det}(H_3) = -4 - 2c_1^2 - 2c_3^2 - \frac{3}{4} c_1^2 c_3^2
	\end{equation}
which is clearly negative for $c_1, c_3 \in [0,2]$.  
Thus, all eigenvalues of $H_3$ are negative, and therefore all eigenvalues of the holomorphic block are negative for $r$ sufficiently large.  

Note that the entries of $H_3$ when $c_1 = c_3 = 2$ are $-4$ along the diagonal and $-2$ off of the diagonal.  
A direct calculation shows that the eigenvalues of this matrix are $-2$ (with multiplicity $2$) and $-8$.  
The eigenvalue of $-8$ has an eigenvector of $\begin{bmatrix}1& 1 & 1 \end{bmatrix} = (Y_1 \wedge Y_2) + (Y_3 \wedge Y_4) + (Y_5 \wedge Y_6)$.  
When $c_1 = c_3 = 0$ the matrix is diagonal with eigenvalues of $-1$ (with multiplicity 2) and $-4$.  
At this point one can start to see that, as each $c_i$ varies from $2$ to $0$, the two eigenvalues of $-2$ increase to $-1$ while the remaining eigenvalue of $-(2n+2)$ increases to $-4$.  

There are six additional blocks in the matrix representation $[\cR]_{\beta}$.  
The blocks corresponding to the pairs $(Y_1 \wedge Y_5, Y_2 \wedge Y_6)$, $(Y_1 \wedge Y_6, Y_2 \wedge Y_5)$, $(Y_3 \wedge Y_5, Y_4 \wedge Y_6)$, and $(Y_3 \wedge Y_6, Y_4 \wedge Y_5)$ are all identical to what was considered in the $n=2$ case (except with $c_1$ replaced with $c_3$ in the latter two cases).  

The remaining two blocks correspond to the pairs $(Y_1 \wedge Y_3, Y_2 \wedge Y_4)$ and $(Y_1 \wedge Y_4, Y_2 \wedge Y_3)$.  
These blocks did not appear in the $n=2$ case since the horizontal fiber $\R^{2n-2} = \R^2$ was not large enough to contain two distinct holomorphic pairs.  
For $r$ large, this block can be approximated by
	\begin{equation}\label{eqn:block 2}
	\begin{bmatrix}
	-1 & - \frac{1}{4} c_1 c_3  \\
	- \frac{1}{4} c_1 c_3 & -1
	\end{bmatrix}
	\end{equation}
The characteristic polynomial factors as
	\begin{equation*}
	\left( \l - \left( -1 + \frac{1}{4} c_1 c_3 \right) \right) \left( \l - \left( -1 - \frac{1}{4} c_1 c_3 \right) \right).
	\end{equation*}
From here one sees immediately that both eigenvalues are nonpositive, and that the eigenvalues when $c_1 = c_3 = 2$ are $-2$ and $0$.    
These are the same eigenvalues that come from the previous four blocks which makes sense:  since the complex hyperbolic metric is symmetric it should not matter if the holomorphic pairs both come from the horizontal fiber or if one of them comes from the vertical fiber.  
The eigenvalues when $c_1 = c_3 = 0$ are both (approximately) $-1$.

Since one of the eigenvalues is again 0, we must consider the case for general r.  
For all values of $r$ this block has the form
	\begin{equation*}
	\begin{bmatrix}
	-1 & \frac{-c_1 c_3 \sinh^2(r) - 4}{4 \cosh^2(r)}  \\
	\frac{-c_1 c_3 \sinh^2(r) - 4}{4 \cosh^2(r)} & -1
	\end{bmatrix}
	\end{equation*}
This case is simpler than the other blocks in that the characteristic polynomial factors as
	\begin{equation*}
	\left( \l - \left( -1 + \frac{-c_1 c_3 \sinh^2(r) - 4}{4 \cosh^2(r)} \right) \right) \left( \l - \left( -1 - \frac{-c_1 c_3 \sinh^2(r) - 4}{4 \cosh^2(r)} \right) \right)
	\end{equation*}
One can see that the larger eigenvalue, which comes from the right-hand term, is nonpositive:
	\begin{align*}
	-1 - \frac{-c_1 c_3 \sinh^2(r) - 4}{4 \cosh^2(r)} &\leq 0  \\
	\Longleftrightarrow \quad \frac{c_1 c_3 \sinh^2(r) + 4}{4 \cosh^2(r)} &\leq 1  \\
	\Longleftrightarrow \qquad \hskip 9pt c_1 c_3 \sinh^2(r) &\leq 4\cosh^2(r) - 4 = 4 \sinh^2(r)  \\
	\Longleftrightarrow \qquad \hskip 46pt c_1 c_3 &\leq 4.
	\end{align*}

\subsection{The case for general $n$}
We formalize the arguments from the previous two Subsections to prove the following.  

\vskip 10pt

\begin{theorem}\label{thm: structure constant operator}
The metric $\mu$ from equation \eqref{eqn:mu} has nonpositive curvature operator for the values $h(r) = \cosh(r)$, $v(r) = \sinh(2r)$, $c_i \in [0,2]$ for all $i$, and for $r$ sufficiently large.  
Moreover, when $c_i = 2$ for all $i$, the metric $\mu$ is equal to the complex hyperbolic metric $\c_n$ and has eigenvalues
	\begin{itemize}
	\item  $0$ with multiplicity $n^2 - n$.  
	\item  $-2$ with multiplicity $n^2 - 1$.  
	\item  $-(2n+2)$ with multiplicity $1$.  
	\end{itemize}
Lastly, the eigenspace corresponding to the eigenvalue of $-(2n+2)$ is equal to the span of the vector
	$$ \sum_{k=0}^{n} \left( Y_{2k-1} \wedge Y_{2k} \right) $$
where the basis $(Y_i)$ is as in \eqref{eqn:ON basis}.  
\end{theorem}

\begin{proof}
The metric $\mu$ is defined on the space $X = \R^{2n-2} \times \S^1 \times (0, \infty)$.  
Let $q = (p, \theta, r) \in X$ and, as in Section \ref{section2: polar coordinates}, choose a holomorphic frame $(X_i)_{i=1}^{2n-2}$ about $q$.
Recall the values of the Lie brackets from equation \eqref{eqn:structure constants 2}.
Let
	\begin{equation*}
	Y_i = \frac{1}{\cosh(r)} X_i \text{ for } i=1, \hdots, 2n-2 \qquad Y_{2n-1} = \frac{1}{\sinh(r) \cosh(r)} \ddt \qquad Y_{2n} = \ddr.
	\end{equation*}
Then $(Y_i)_{i=1}^{2n}$ is an orthonormal basis for $\mu$.  

To calculate the eigenvalues for the curvature operator of $\mu$ we order the basis in the same manner as in the $n=3$ case.  
The blocks along the diagonal come in three forms, the same forms that arise when $n=3$.  
The only block that is different is the holomorphic block, which is size $n \times n$ and corresponds to the holomorphic pairs $Y_i \wedge Y_{i'}$.   
The other two types of blocks are identical to what was considered when $n=2$ or $3$.  
Recall that, when $c_i = 2$ for all $i$, each of these $2 \times 2$ blocks contributed an eigenvalue of -2 and 0.  
There are
	\begin{equation*}
	\frac{1}{2} \left( {2n \choose 2} - n \right) = n^2 - n
	\end{equation*}
of these blocks.  
And so the curvature operator $\cR$ for the complex hyperbolic metric inherits eigenvalues of -2 and 0, each with multiplicity $n^2 - n$, from these blocks.  

To prove the first part of Theorem \ref{thm: structure constant operator} we need to show that the holomorphic block $H_n$ is negative definite.  
For $r$ large this symmetric $n \times n$ matrix can be approximated by
	\begin{equation*}
	H_n = \begin{bmatrix}
	-1 - \frac{3}{4} c_1^2 & - \frac{1}{2} c_1 c_3 & - \frac{1}{2} c_1 c_5 & \hdots & - \frac{1}{2} c_1 c_{2n-3} & -c_1  \\
	- \frac{1}{2} c_1 c_3 & -1 - \frac{3}{4} c_3^2 & - \frac{1}{2} c_3 c_5 & \hdots & - \frac{1}{2} c_3 c_{2n-3} & - c_3  \\
	- \frac{1}{2} c_1 c_5 & - \frac{1}{2} c_3 c_5 & -1 - \frac{3}{4} c_5^2 & \hdots & - \frac{1}{2} c_5 c_{2n-3} & - c_5  \\
	\vdots & \vdots & \vdots & \ddots & \vdots & \vdots  \\
	- \frac{1}{2} c_1 c_{2n-3} & - \frac{1}{2} c_3 c_{2n-3} & - \frac{1}{2} c_5 c_{2n-3} & \hdots & -1 - \frac{3}{4} c_{2n-3}^2 & -c_{2n-3}  \\
	-c_1 & -c_3 & -c_5 & \hdots & -c_{2n-3} & -4
	\end{bmatrix}
	\end{equation*}
To show that $H_n$ is negative definite we instead show that $-H_n$ is positive definite.  
The submatrix obtained by removing the first row and column of $-H_n$ is $- H_{n-1}$, up to renaming the structure constants.  
So, by induction, we just need to show that the sign of the determinant of $-H_n$ is positive for all $c_i \in [0,2)$ (the determinant is $0$ when $c_i = 2$ for all $i$).  

We can give an exact description of this determinant.  
Considering $c_1^2, c_3^2, \hdots , c_{2n-3}^2$ as independent variables, the determinant of $- H_n$ has degree $n-1$ and, within each term of this polynomial, every variable (ie, each $c_i^2$) appears with power either zero or one.  
If a given term has degree $k$, meaning that $k$ of the $c_{i}^{2}$'s appear in this term, the coefficient is $\ds{ \frac{4(k+1)}{4^k} }$.  

So, for example, if $n=2$ then det$(-H_2$) has degree 1.  
The constant term is $\ds{ \frac{4(0+1)}{4^0} = 4}$, and the coefficient of the degree 1 term is $\ds{ \frac{4(1+1)}{4^1} = 2 }$.  
Thus, det($-H_2$) = $2c_1^2 + 4$ in agreement with \eqref{eqn:detS2}.  
The determinant of $-H_3$ has degree 2.  
The coefficients of the constant and degree 1 terms are the same as above.  
The degree 2 term has coefficient $\ds{ \frac{4(2+1)}{4^2} = \frac{3}{4} }$.  
Therefore, 
	\begin{equation*}
	\text{det} (-H_3) = \frac{3}{4} c_1^2 c_3^2 + 2c_1^2 + 2c_3^2 + 4 
	\end{equation*}
which coincides with \eqref{eqn:detS3}.  

Proving the above formula for the determinant of $-H_n$ via induction is difficult, but a simpler argument can be used to show that $-H_n$ is has positive determinant.  
Consider cofactor expansion along the top row of $-H_n$.  
This determinant has the form
	\begin{equation*}
	\left( 1 + \frac{3}{4} c_1^2 \right) (\text{det}(-H_{n-1})) + c_1^2 ( \text{sum of determinant of minors removing the top row} )
	\end{equation*}
where the $c_1^2$ in the second term comes from the fact that every entry in the top row of $H_n$ has a $c_1$ factor, and in each minor we can factor a $c_1$ out of the first column.  
Taking the partial derivative with respect to $c_1$ and then applying basic calculus shows that the only critical point of this determinant, when considered as a function of $c_1, c_3, \hdots, c_{2n-3}$, must satisfy $c_1 = 0$.  
By symmetry, we then must have that $c_i = 0$ for all $i$.  
But we have already shown that $\text{det}(-H_n) = 4$ when $c_i = 0$ for all $i$.  
Then, since this same determinant is $0$ when $c_i = 2$ for all $i$, we see that all values for this determinant must fall in the interval $[0,4]$ for the specified values of $c_i$.  

Finally, we summarize the eigenvalues of the curvature operator with respect to the standard complex hyperbolic metric.  
When $c_i = 2$ for all $i$, the matrix $H_n$ has $-4$ for each diagonal entry and $-2$ for each off-diagonal entry.  
One can see directly that this matrix has eigenvalues of $-2$ (with multiplicity $n-1$) and $-(2n+2)$ (with multiplicity $1$) as follows.  
The matrix $H_n + 2I_n$ is the constant matrix where every entry is $-2$. 
The dimension of the kernel of this matrix is clearly $n-1$, proving that $-2$ is an eigenvalue with multiplicity $n-1$.  
Finally, if one multiplies $H_n$ by the vector $\begin{bmatrix}
1 & 1 & 1 & \hdots & 1 \end{bmatrix}$, they obtain the vector $\begin{bmatrix}
-(2n+2) & -(2n+2) & -(2n+2) & \hdots & -(2n+2) \end{bmatrix}$.  
This shows that $-(2n+2)$ is an eigenvalue of $H_n$ with multiplicity $1$.  

Therefore, the eigenvalues of the complex hyperbolic metric $\c_n$ are
	\begin{itemize}
	\item  0 with multiplicity $n^2 - n$
	\item  $-2$ with multiplicity $(n^2 - n) + (n-1) = n^2 - 1$
	\item  $-(2n+2)$ with multiplicity $1$.  
	\end{itemize}

\end{proof}

\vskip 20pt

\section{Overview of the metric and proof of the main theorem}\label{section4: proof of main theorem}

Following \cite{ST}, let $\Gamma' < \text{PU}(n,1)$ be a cocompact congruence arithmetic lattice of simple type, and let $M' = \Gamma' \setminus \C \H^n$.
Then, by \cite{ST}, \cite{DKV}, and the residual finiteness of $\Gamma$ (see the Introduction to \cite{Min CH}), for any $R > 0$ there exists $\Gamma < \Gamma'$ and an integer $d > 2$ such that
	\begin{itemize}
	\item  the manifold $M = \Gamma \setminus \C \H^n$ contains an embedded (possibly disconnected) complex codimension 1 totally geodesic submanifold $N$.
	\item  the $d$-fold ramified branched covering $X$ of $M$ about $N$ is a smooth manifold.  Let $\pi: X \to M$ denote this branched covering.  
	\item  the normal injectivity radius of the ramification locus  $\pi^{-1}(N)$ within $X$ is at least $R$.  
	\end{itemize}
This manifold $X$ is the manifold that is referred to in Theorem \ref{thm:main theorem}.  
It is not homotopy equivalent to a quotient of $\C \H^n$ by \cite{ST}, and it is K\"{a}hler by \cite{Zheng}.  

Let $\tilde{X}$ denote the universal cover of $X$.
The preimage of $N$, denoted $\tilde{N}$, within $\tilde{X}$ is a disconnected collection of hyperplanes.
Each hyperplane is isometric to $\C \H^{n-1}$ with respect to the pullback metric on $\tilde{X}$ which we denote $\c_n^*$.
To define the metric $g$ on $X$ we perform ``geometric surgery" on $\c_n^*$ within the $R$-tube about each component of $\tilde{N}$.  
As long as we arive back at the metric $\c_n^*$ by the end of the $R$-tube, our new metric $g$ will descend to a Riemannian metric on $X$.  

An outline of the construction of the metric $g$ from \cite{Min CH} is as follows.  
The $R$-tube about each component of $\tilde{N}$ is diffeomorphic to $E \times [0,R)$ where $E = \R^{2n-2} \times \S^1$.  
We define
	\begin{itemize}
	\item  $g = \c_n$ on $E \times [0,r_1)$ for some $r_1 >> 0$. \\
	 
	\item  over $E \times [r_1, r_2]$ for $r_2 >> r_1$ we slowly ``unwind" the horizontal distribution spanned by the holomorphic frame $\{ X_1 , \hdots , X_{2n-2} \}$ until it is tangent to the $r$-tube about $\R^{2n-2}$.  
	This amounts to very slowly decreasing the structure constants from equation \eqref{eqn:structure constants 2} from $2$ to $0$ for each $i$.  
	Since the beginning structure constant $c_i = 2$ is independent of the point $p \in \R^{2n-2}$ and angle $\theta \in \S^1$, this can be done in an invariant manner over $E$.  
	There are no holonomy issues since this procedure is independent of $\theta$.  
	The curvature equations in Theorem \ref{thm:new curvature equations} are valid provided $\ds{ \left[ X_i , \partial / \partial r \right] = 0 }$ for all $i$.  
	This will fail to be true as the vector fields turn with the horizontal distribution.  
	In general, $\ds{ \left[ X_i , \partial / \partial r \right] = d_i (\partial / \partial r) }$ for some $d_i \in \R$.  
	But, by choosing $r_2$ sufficiently large, we can keep $|d_i| < \delta$ for any prescribed $\delta > 0$.  
	As one chooses smaller values of $\d$, the curvature formulas in \ref{thm:new curvature equations} provide a better approximation of the sectional curvature tensor of $g$.  
	Note that on $E \times \{r_2 \}$ the metric $g$ equals the integrable complex hyperbolic metric $g_I$.  \\
	
	\item  on $E \times [r_2, r_3]$ for $r_3 > > r_2$ we slowly increase the angle about $\R^{2n-2}$ in a nearly identical way as to what was done by Gromov and Thurston in \cite{GT}.  
	By increasing the angle at a sufficiently slow rate, we can keep all sectional curvatures of $g$ pinched near $[-4,-1]$ (see Theorem 3.5 of \cite{Min CH}).
	Note that the metric $g$ on $E \times \{ r_3 \}$ will have total angle of $2 \pi d$ about $\R^{2n-2}$ and, when restricted to an arc of the form $\R^{2n-2} \times A_i \times \{ r_3 \}$ where $A_i$ has total angle $2 \pi$ with respect to $g$, the metric $g$ is again equal to $g_I$.  \\
	
	\item  Subdivide $\S^1$ into $d$ arcs $\{ A_1, \hdots , A_d \}$ of length $2\pi / d$ with respect to the standard metric on $\S^1$.  
	We then ``rewind" the metric $g$ over $\R^{2n-2} \times A_i \times [r_3, R]$ from $g_I$ to $\c_n$ in an exactly backwards manner as to what was described in the second bullet point above.  
	Since this process was invariant of $p \in \R^{2n-2}$ and $\theta \in \S^1$ it agrees at the intersection of any two arcs.  
	The resulting metric on $E \times \{ R \}$ is equal to the pullback metric $\c_n^*$.  
	\end{itemize}

With this description complete we can now prove Theorems \ref{thm:main theorem} and \ref{thm:main cor}.

\begin{proof}[Proof of Theorem \ref{thm:main theorem}]
Let $\cR$ denote the curvature operator of the metric $g$ described above.  
The fact that $\cR$ is nonpositive over $E \times [0,r_1]$ and for $r \geq R$ is given by Theorem \ref{thm:ch eigenvalues}.  
Theorem \ref{thm: structure constant operator} shows that $\cR$ is nonpositive on $E \times [r_1, r_2]$ and over $E \times [r_3, R]$.  

All that is left to check is the region $E \times [r_2, r_3]$.  
Over this region all structure constants in \eqref{eqn:structure constants 2} are zero and, via Theorem \ref{thm:new curvature equations}, all mixed terms of $g$ are approximately $0$ for $r_2 >  > 0$ (and with $h(r) = \cosh(r)$ and $v(r) = \sinh(2r)$).  
Thus the basis $(Y_i \wedge Y_j)$ with $i < j$ and $Y_i, Y_j$ ranging over all vectors in \eqref{eqn:ON basis} (approximately) diagonalizes $\cR$.  
Since this metric has negative sectional curvature for $r_3 - r_2$ chosen sufficiently large, all eigenvalues of $\cR$ are negative over this region.  
\end{proof}

\begin{proof}[Proof of Theorem \ref{thm:main cor}]
Consider the second bullet point above in the construction of the metric $g$.  
To construct the metric for Theorem \ref{thm:main cor} we instead first turn the horizontal distribution slightly away from the $r$-tube about $\R^{2n-2}$ creating structure constants $c_i = 2 + \d$ for a sufficiently small $\d$.  
We then slowly turn the horizontal distribution back toward the $r$-tube as in the construction above.  
A quick analysis of the blocks \eqref{eqn:block 1} and \eqref{eqn:block 2} show that both have a positive eigenvalue when $c_i > 2$ for all $i$.  
Thus, the curvature operator $\cR$ will not be nonpositive.  
However, for $\d > 0$ chosen sufficiently small, all sectional curvatures will still be negative (and, moreover, will be pinched near $[-4,-1]$).  
This follows from Lemma \ref{thm:pinched curvature} where the choice for the structure constants $c_i$ can be extended slightly beyond $[-2,2]$ by continuity.  
\end{proof}

\end{document}